\documentclass[12pt, reqno]{amsart}
\usepackage{amssymb, amsmath, amsthm, graphicx}
\usepackage[colorlinks=true, citecolor=blue, urlcolor=blue]{hyperref}

\newtheorem{thrm}{Theorem}
\newtheorem{lemma}{Lemma}
\newtheorem{probl}{Problem}

\title[On the existence of two affine-equivalent frameworks]{On the existence of two affine-equivalent frameworks with prescribed edge lengths in Euclidean $d$-space} 
\author{Victor Alexandrov}
\address{Sobolev Institute of Mathematics, Koptyug ave., 4, Novosibirsk, 630090, Russia
and Department of Physics, Novosibirsk State University, Pirogov str., 2, Novosibirsk,
630090, Russia\newline 
\indent \href{https://orcid.org/0000-0002-6622-8214}{\rm{ORCID iD: 0000-0002-6622-8214}}}
\email{alex@math.nsc.ru}
\date{June 24, 2023}

\begin{document}

\begin{abstract}
We study the problem of existence of two affine-equivalent
bar-and-joint frameworks in Euclidean $d$-space
which have prescribed combinatorial structure and edge lengths.
We prove that theoretically this problem is always solvable, but we cannot
propose any practical algorithm for its solution.
\end{abstract}
\maketitle

\section{Introduction}\label{sec1}

Given a graph $G$, we denote by $E(G)$ and $V(G)$ the sets of its edges
and vertices respectively.

A \textit{bar-and-joint framework} in $\mathbb{R}^d$, $d\geqslant 1$,  
is a graph $G$ and a mapping ${\bf p}: G\to \mathbb{R}^d$ which assigns a 
point ${\bf p}_i$ in $\mathbb{R}^d$ to each vertex $i$ of $G$ and is such that
${\bf p}_i\neq {\bf p}_j$ for all $(i,j)\in E(G)$.
The mapping $\bf{p}$ is called a \textit{configuration};
the point ${\bf p}_i$ is called its \textit{node}; and
the straight-line segment with end-points ${\bf p}_i$, ${\bf p}_j$ is called 
a \textit{bar} provided that $(i,j)$ is an edge of $G$.
A bar-and-joint framework is denoted by $G(\bf{p})$.
This terminology and notation is standard in the rigidity theory of 
bar-and-joint frameworks, see, e.\,g., \cite{CG22} and references given there.

In this article, we treat the configuration $\bf{p}$ as a geometric realization of $G$ and often abbreviate a ``bar-and-joint framework'' to a ``framework''.

We address ourselves to the following problem:

\begin{probl}\label{probl1}
Let $d\geqslant 1$ be a natural number, let $G$ be a graph, and 
let $\lambda, \lambda':E(G)\to (0,+\infty)$ be two functions
which assign real numbers $\lambda_{ij}$ and $\lambda'_{ij}$ to each 
edge $(i,j)$ of $G$.
Do there exist two frameworks $G(\bf{p})$ and $G(\bf{p'})$ in $\mathbb{R}^d$
such that 

{\em (a)} $|{\bf p}_i-{\bf p}_j|=\lambda_{ij}$ for every $(i,j)\in E(G)$,

{\em (b)} $|{\bf p'}_{\!\! i}-{\bf p'}_{\!\! j}|=\lambda'_{ij}$ 
for every $(i,j)\in E(G)$,

{\em (c)} $G(\bf{p})$ and $G(\bf{p'})$ are affine-equivalent, 
i.e., that there is an affine transformation of $\mathbb{R}^d$ such that 
$A({\bf p}_i)={\bf p'}_{\!\! i}$ for every $i\in V(G)$,

{\em (d)} $\operatorname{aff}({\bf p}(G))= \mathbb{R}^d$, where 
$\operatorname{aff}({\bf p}(G))$ is the affine hull of the set 
${\bf p}(G)= \{{\bf p}_i\in \mathbb{R}^d : i\in V(G)\}$?
\end{probl}

In other words, Problem~\ref{probl1} asks if we can find two affine-equivalent 
geometric realizations of $G$ with prescribed lengths of bars. 

Problem~\ref{probl1} generalizes the problem of recognition of affine-equivalent 
polyhedra in Euclidean 3-space via their natural developments.
The latter is studied in \cite{Al22}, \cite{Al23} and is motivated by the 
Cauchy Rigidity Theorem \cite{Al05}.
The arguments used in \cite{Al22}, \cite{Al23} rely significantly on the structure
of polyhedral surfaces in Euclidean 3-space and, thus, cannot be directly 
applied to the study of Problem~\ref{probl1}.
In this article, we propose an algorithm that solves Problem~\ref{probl1} for any input 
data, i.\,e., for any dimension $d$, any graph $G$, and any sets of prescribed 
lengths $\{\lambda_{ij}\}_{(i,j)\in E(G)}$, $\{\lambda'_{ij}\}_{(i,j)\in E(G)}$
of the bars.
The algorithm proposed give the assurance that Problem~\ref{probl1} is always solvable, 
rather than provides us with an effective tool for its practical solution.

\section{Prior results}\label{sec2} 

If $X$ is a set, we denote by $|X|$ the cardinality of $X$.
If $x$ is a vector in $\mathbb{R}^d$, we denote by $|x|$ the length of $x$.

A \textit{semimetric space} is a pair $(X,\rho)$, where $X$ is a set and 
$\rho:X\times X\to \mathbb{R}$ is a function such that, for all $x,y\in X$, 
the following properties hold true:  $\rho(x,y)=\rho(y,x)$,
$\rho(x,y)\geqslant 0$, and $\rho (x,y)=0$ if and only if $x=y$, 
see \cite[Definition 5.1]{Bl70}.
Such $\rho$ is called a \textit{semimetric} on $X$.

Given a semimetric space $(X,\rho)$ and its finite subset $Y=\{y_0, y_1,\dots, y_k\}$,
we put by definition $\rho_{ij}=\rho(y_i, y_j)$ and
\begin{equation}\label{eqn1}
\operatorname{cmd}(Y)\stackrel{\textrm{def}}{=}
\left|
\begin{array}{ccccc}
0 & 1             & 1             & \dots & 1              \\
1 & 0             & {\rho}^2_{01} & \dots & {\rho}^2_{0k}  \\
1 & {\rho}^2_{10} & 0             & \dots & {\rho}^2_{1k}  \\
. & .             & .             & .     & .              \\
1 & {\rho}^2_{k0} & {\rho}^2_{k1} & \dots & 0 
\end{array}
\right|.
\end{equation} 
The $(k+2)\times(k+2)$ determinant in (\ref{eqn1}) is called the 
\textit{Cayley--Menger determinant} of the set $Y$ or of the points 
$y_0, y_1,\dots, y_k$, see, e.\,g., \cite[Section 40]{Bl70}.

A mapping $f:X\to\mathbb{R}^d$ is called an 
\textit{isometric embedding} of a semimetric space $(X,\rho)$ into $\mathbb{R}^d$,
$d\geqslant 1$, if $|f(x)-f(y)|=\rho(x,y)$ for all $x,y\in X$.

\begin{thrm}[K. Menger]\label{thrm1}
Suppose $(X,\rho)$ is a semimetric space and $d\geqslant 1$ is an integer.
An isometric embedding $f:X\to\mathbb{R}^d$ such that 
$\operatorname{aff} f(X) = \mathbb{R}^d$ does exist 
if and only if the following statements hold true:
\begin{enumerate}
\renewcommand{\theenumi}{\roman{enumi}}
\item  $|X|\geqslant d+1$,
\item $(-1)^{|Y|}\operatorname{cmd}(Y)\geqslant 0$ for every $Y \subset X$
	such that $|Y|\leqslant d+1$,
\item $(-1)^{d+1}\operatorname{cmd}(Y_0)> 0$ for some $Y_0\subset X$ such that
	$|Y_0|=d+1$,
\item $\operatorname{cmd}(Y)=0$ for every $Y\subset X$ such that $|Y|=d+2$.  
\end{enumerate}
\end{thrm}
 
We have formulated this theorem, first proved by K. Menger in 1928, 
in a form convenient for our purposes. 
We refer to \cite[Chapter IV]{Bl70} for its proof, and to \cite{HL19}
for its applications and alternative formulations.

Below we will make use of the following well-known fact:
If $X=\{x_0, x_1, \dots ,x_k\}\subset \mathbb{R}^d$,
$1\leqslant k\leqslant d$, and $\rho_{ij}=|x_i-x_j|$ for all $i,j=0,1,\dots,k$,
the $k$-dimensional volume $\operatorname{vol}_k (X)$ of the 
simplex $X\subset\mathbb{R}^d$ is related to the Cayley--Menger determinant of 
$X\subset (X,\rho)$ by the following formula
\begin{equation}\label{eqn2}
[\operatorname{vol}_k(X)]^2=
\frac{(-1)^{k+1}}{2^k k!}
\operatorname{cmd}(X),
\end{equation}
see, e.g., \cite[Section 40]{Bl70}.

\section{Auxiliary result}\label{sec3}

Suppose $d\geqslant 1$ is an integer, 
$K_{d+2}$ is the complete graph on $d+2$ vertices $i=0, 1, \dots, d+1$,
and ${\bf p}: {K_{d+2}}\to \mathbb{R}^d$ is a configuration such that
the dimension of the affine hull $\operatorname{aff}(\Delta)$ of the set 
$\Delta=\{{\bf p}_1, {\bf p}_2, \dots, {\bf p}_d\}$ is equal to $d-1$.
Define the metric $\rho$ on $K_{d+2}$ by the rule $\rho_{ij}=|{\bf p}_i-{\bf p}_j|$.

In this Section, our goal is to learn to distinguish between the following three cases 
using only the metric space $(K_{d+2}, \rho)$:
\begin{enumerate}
\renewcommand{\theenumi}{\Alph{enumi}}
\item ${\bf p}_0$ and ${\bf p}_{d+1}$ lie in the same open half-space 
bounded by the hyperplane $\operatorname{aff}(\Delta)$, 
\item ${\bf p}_0$ and ${\bf p}_{d+1}$ lie in $\operatorname{aff}(\Delta)$, 
\item ${\bf p}_0$ and ${\bf p}_{d+1}$ lie in different open half-spaces bounded by 
$\operatorname{aff}(\Delta)$.
\end{enumerate}

Observe that $\operatorname{cmd} (K_{d+1})$ is a quadratic polynomial with respect
to the variable $t=\rho^2_{0,d+1}$, whose coefficients are polynomials
with respect to $\rho_{ij}$, $\{i,j\}\neq \{0,d+1\}$.
Denote this polynomial as $Ut^2+Vt+W$.
It is straightforward to see that $U=-\operatorname{cmd}(\Delta)$.
Then, from (\ref{eqn2}) it follows that
$U=(-1)^{d-1}2^{d-1}(d-1)![\operatorname{vol}_{d-1}(\Delta)]^2$.
Hence, $U<0$ if $d$ is even, and $U>0$ if $d$ is odd.

\begin{lemma}\label{lemma1}
Using the notation introduced above, we can assert that

{\scriptsize $\bullet$} the case {\em (A)} holds true if and only if
\begin{equation}\label{eqn3}
(-1)^d(2U\rho^2_{0,d+1}+V)>0,
\end{equation}

{\scriptsize $\bullet$} the case {\em (B)} holds true if and only if 
\begin{equation}\label{eqn4} 
2U\rho^2_{0,d+1}+V=0,
\end{equation}

{\scriptsize $\bullet$} the case {\em (C)} holds true if and only if 
\begin{equation}\label{eqn5}
(-1)^d(2U\rho^2_{0,d+1}+V)<0.
\end{equation} 
\end{lemma}

\begin{proof}
Since $(K_{d+2}, \rho)$ embeds isometrically in $\mathbb{R}^d$,
the condition (iv) of Theorem~\ref{thrm1} holds true. 
Hence, $t_1=\rho^2_{0,d+1}$ is a root of $Ut^2+Vt+W$.

Suppose the case (A) holds true. Then, $Ut^2+Vt+W$ has one more 
positive root $t_2$, which corresponds to the configuration obtained from
${\bf p}$ by replacement of ${\bf p}_{d+1}$ by its image under the reflection
in the hyperplane $\operatorname{aff}(\Delta)$. This implies $t_1<t_2$. 
Thus, $2\rho^2_{0, d+1}=2t_1<t_1+t_2=-V/U$, i.\,e., (\ref{eqn3}) holds true.
Arguing backwards, we see that (\ref{eqn3}) implies (A).

Suppose the case (B) holds true. Then, $t_1$ is a root of multiplicity 2 of 
$Ut^2+Vt+W$. Thus, $2\rho^2_{0, d+1}=2t_1=t_1+t_2=-V/U$, i.\,e., (\ref{eqn4}) holds true. 
Arguing backwards, we see that (\ref{eqn4}) implies (B).

At last, suppose the case (C) holds true. Then, as before, $Ut^2+Vt+W$ has one more 
positive root $t_2$, which corresponds to the configuration obtained from
${\bf p}$ by replacement of ${\bf p}_{d+1}$ by its image under the reflection
in the hyperplane $\operatorname{aff}(\Delta)$. This implies $t_2<t_1$. 
Thus, $2\rho^2_{0, d+1}=2t_1>t_1+t_2=-V/U$, i.\,e., (\ref{eqn5}) holds true.
Arguing backwards, we see that (\ref{eqn5}) implies (C).
\end{proof}

\section{Main result}\label{sec4}

In this Section, we prove that a positive solution of Problem~\ref{probl1} is equivalent 
to the existence of a solution to a system of polynomial equalities and inequalities.
We start with a description of that system.

In accordance with the notation adopted in Problem~\ref{probl1},
suppose we are given a natural number $d\geqslant 1$, 
a graph $G$, and two functions $\lambda, \lambda':E(G)\to (0,+\infty)$.
With each pair of vertices $i,j\in V(G)$, even those that are not connected by an edge
of $G$, we associate two real variables $z_{ij}$ and $z'_{ij}$.
At the end of our solution of Problem~\ref{probl1} we will see that,
if the affine-equivalent configurations
${\bf p}, {\bf p'}: G\to \mathbb{R}^d$ exist, then
$z_{ij}=|{\bf p}_i-{\bf p}_j|^2$ and
$z'_{ij}=|{\bf p'}_{\!\! i}-{\bf p'}_{\!\! j}|^2$ for all $(i,j)\in V(G)$.

For every $I=\{i_0,i_1,\dots, i_k\}\subset V(G)$, we put by definition
$Z_I=\{z_{ij}: i,j\in I\}$, $Z'_I=\{z'_{ij}: i,j\in I\}$,
\begin{equation*}
{\operatorname{acmd}(Z_I)\stackrel{\textrm{def}}{=}
\left|
\begin{array}{ccccc}
0 & 1          & 1          & \dots & 1         \\
1 & 0          & z_{i_0i_1} & \dots & z_{i_0i_k}  \\
1 & z_{i_1i_0} & 0          & \dots & z_{i_1i_k}  \\
. & .          & .          & .     & .         \\
1 & z_{i_ki_0} & z_{i_ki_1} & \dots & 0 
\end{array}
\right|,}
\quad \textrm{and} \quad
{\operatorname{acmd}(Z'_I)\stackrel{\textrm{def}}{=}
\left|
\begin{array}{ccccc}
0 & 1           & 1           & \dots & 1         \\
1 & 0           & z'_{i_0i_1} & \dots & z'_{i_0i_k}  \\
1 & z'_{i_1i_0} & 0           & \dots & z'_{i_1i_k}  \\
. & .           & .           & .     & .         \\
1 & z'_{i_ki_0} & z'_{i_ki_1} & \dots & 0 
\end{array}
\right|.}
\end{equation*} 
We call $\operatorname{acmd}(Z_I)$ and $\operatorname{acmd}(Z'_I)$ the 
\textit{abstract Cayley--Menger determinants} of $Z_I$ and $Z'_I$ respectively.

Suppose $I=\{i_0,\dots, i_{d+1}\}\subset V(G)$.
Note that, for every $0\leqslant r\leqslant d$, 
$\operatorname{acmd}(Z_I)$ is a quadratic polynomial with respect 
to the variable $z_{i_r,i_{d+1}}$ whose coefficients are polynomials
with respect to the variables $z_{i_m,i_n}$, $\{i_m,i_n\}\neq \{i_r,i_{d+1}\}$.
Denote this polynomial by 
$U_{I,r}(z_{i_r,i_{d+1}})^2+V_{I,r}z_{i_r,i_{d+1}}+W_{I,r}$.
Similarly, $\operatorname{acmd}(Z'_I)$ is a quadratic polynomial in 
$z'_{i_r,i_{d+1}}$ whose coefficients are polynomials in $z'_{i_mi_n}$, 
$\{i_m,i_n\}\neq \{i_r,i_{d+1}\}$.
Denote this polynomial by 
$U'_{I,r}(z'_{i_r,i_{d+1}})^2+V'_{I,r}z'_{i_r,i_{d+1}}+W'_{I,r}$.

Now we can explicitly write the system of polynomial equalities and inequalities
(\ref{eqn6})--(\ref{eqn12}), involved in our solution of Problem~\ref{probl1}:
\begin{equation}\label{eqn6}
z_{ij}\geqslant 0 \textrm{\ and\ } z'_{ij}\geqslant 0 \quad 
\textrm{for all\ } i,j\in V(G);
\end{equation} 
\begin{equation}\label{eqn7}
z_{ij}=(\lambda_{ij})^2  \textrm{\ and\ } z'_{ij}=(\lambda'_{ij})^2 \quad 
\textrm{for all\ } (i,j)\in E(G);
\end{equation}
\begin{equation}\label{eqn8}
(-1)^{|I|}\operatorname{acmd}(Z_I)\geqslant 0 \textrm{\ and\ }
(-1)^{|I|}\operatorname{acmd}(Z'_I)\geqslant 0
\textrm{\ for every\ } I \subset V(G) \textrm{\ such that\ } |I|\leqslant d+1;
\end{equation}
\begin{equation}\label{eqn9}
\textrm{there is\ } I_*=\{i_0,\dots, i_{d}\}\subset V(G) \textrm{\ such that\ }
\operatorname{acmd}(Z_{I_*})\neq 0;
\end{equation}
\begin{equation}\label{eqn10}
\operatorname{acmd}(Z_I)= \operatorname{acmd}(Z'_I) =0
\ \textrm{for every} \ I\subset V(G), \ |I|=d+2;
\end{equation}
\begin{equation}\label{eqn11}
\textrm{there is} \ \alpha>0 \ \textrm{such that} \ 
\operatorname{acmd}(Z'_I)=\alpha \operatorname{acmd}(Z_I) 
\ \textrm{for every} \ I\subset V(G), \ |I|=d+1;
\end{equation}
\begin{equation}\label{eqn12}
\left.\begin{split}
\textrm{\ for} & \textrm{\ every\ } 
i_{d+1}\in V(G)\setminus I_* \textrm{\ and every\ }
1\leqslant r\leqslant d \\
& \textrm{\ either\ }
[V_{I_*\cup\{i_{d+1}\},r}+2U_{I_*\cup\{i_{d+1}\},r}(z_{i_r,i_{d+1}})^2]
[V'_{I_*\cup\{i_{d+1}\},r}+2U'_{I_*\cup\{i_{d+1}\},r}(z'_{i_r,i_{d+1}})^2]> 0 \\
& \textrm{\ or \ }
[V_{I_*\cup\{i_{d+1}\},r}+2U_{I_*\cup\{i_{d+1}\},r}(z_{i_r,i_{d+1}})^2]^2+
[V'_{I_*\cup\{i_{d+1}\},r}+2U'_{I_*\cup\{i_{d+1}\},r}(z'_{i_r,i_{d+1}})^2]^2= 0.
\end{split}
\right\}
\end{equation}
In (\ref{eqn12}), $I_*\subset V(G)$  is the set whose existence is declared in (\ref{eqn9}).

The main result of the present article is given by the following 

\begin{thrm}\label{thrm2}
Suppose $d\geqslant 1$ is a natural number, suppose $G$ is a graph, and 
suppose $\lambda, \lambda':E(G)\to (0,+\infty)$ are two functions
which assign real numbers $\lambda_{ij}$ and $\lambda'_{ij}$ to each 
edge $(i,j)$ of $G$. 
Then the following two statements are equivalent:
\par
$(\alpha)$ there are two bar-and-joint frameworks $G(\bf{p})$ and $G(\bf{p'})$ in 
$\mathbb{R}^d$ satisfying the conditions {\em (a)--(d)} of the Problem~\ref{probl1};
\par 
$(\beta)$ the system of algebraic equalities and inequalities 
{\em (\ref{eqn6})--(\ref{eqn12})} has a solution in real numbers.
\end{thrm}

\begin{proof}
First we prove the implication $(\alpha)\Rightarrow (\beta)$.
Suppose we are given bar-and-joint frameworks 
$G(\bf{p})$ and $G(\bf{p'})$ in $\mathbb{R}^d$ satisfying the conditions 
(a)--(d) of Problem~\ref{probl1}.

For all $i,j\in V(G)$, put $z_{ij}=|{\bf p}_i-{\bf p}_j|^2$ and
$z'_{ij}=|{\bf p'}_{\!\! i}-{\bf p'}_{\!\! j}|^2$.
For all $(i,j)\in E(G)$, put $\lambda_{ij}=|{\bf p}_i-{\bf p}_j|$ and
$\lambda'_{ij}=|{\bf p'}_{\!\! i}-{\bf p'}_{\!\! j}|$.
With such a choice of the values of $z_{ij}$, $z'_{ij}$ and 
$\lambda_{ij}$, $\lambda'_{ij}$ the relations (\ref{eqn6}), (\ref{eqn7}) 
obviously hold true.
Moreover, 

{\scriptsize $\bullet$} functions $z:G\times G\to \mathbb{R}$ and 
$z':G\times G\to \mathbb{R}$, defined by the formulas $z(i,j)=z_{ij}$ and
$z'(i,j)=z'_{ij}$, are metrics on $G$; 

{\scriptsize $\bullet$} the configurations 
${\bf p}: G\to \mathbb{R}^d$ and ${\bf p'}: G\to \mathbb{R}^d$ are
isometric imbeddings of metric spaces $(G,z)$ and $(G,z')$ respectively; 

{\scriptsize $\bullet$} 
$\operatorname{aff}({\bf p}(G))=\operatorname{aff}({\bf p'}(G))=\mathbb{R}^d$;

{\scriptsize $\bullet$} for every $I=\{i_0,i_1,\dots, i_k\}\subset V(G)$ 
the equalities
\begin{equation}\label{eqn13}
\operatorname{acmd}(Z_I)=\operatorname{cmd}(\{{\bf p}_{i_0},\dots,{\bf p}_{i_k}\})
\quad\textrm{and}\quad
\operatorname{acmd}(Z'_I)=
\operatorname{cmd}(\{{\bf p'}_{\!\! i_0},\dots,{\bf p'}_{\!\! i_k}\})
\end{equation}
are satisfied.

Now we see that the statement (ii) of Theorem~\ref{thrm1}, (\ref{eqn2}), and
(\ref{eqn13}) imply (\ref{eqn8}); 
the statement (iii) of Theorem~\ref{thrm1} and (\ref{eqn13}) imply (\ref{eqn9}); 
the statement (iv) of Theorem~\ref{thrm1} and (\ref{eqn13}) imply (\ref{eqn10}); 
(\ref{eqn11}) holds true with $\alpha=(\operatorname{det}A)^2>0$ 
due to (\ref{eqn13}); and Lemma~\ref{lemma1} implies (\ref{eqn12}).
Thus the implication $(\alpha)\Rightarrow (\beta)$ is proved.

Next, we prove the implication $(\beta)\Rightarrow (\alpha)$.

Suppose $G$ is a graph and suppose the corresponding system of polynomial 
qualities and inequalities (\ref{eqn6})--(\ref{eqn12}) has a solution 
$z_{ij}$, $z'_{ij}$, $\alpha$. 
Obviously, $z=\{z_{ij}\}$ and $z'=\{z'_{ij}\}$ are semimetrics on the complete 
graph $K_{|G|}$.
The conditions (\ref{eqn8})--(\ref{eqn10}) mean that the semimetric space $(K_{|G|},z)$
satisfies the conditions (i)--(iv) of Theorem~\ref{thrm1}.
Hence, there is its isometric embedding ${\bf p}: K_{|G|}\to \mathbb{R}^d$,
${\bf p}_i={\bf p}(i)$, $i\in V(K_{|G|})$,
such that $\operatorname{aff}({\bf p}(K_{|G|}))= \mathbb{R}^d$.
Let $I_*=\{i_0,\dots,i_d\} \subset V(G)= V(K_{|G|})$  be the set whose existence is 
declared in (\ref{eqn9}). Using (\ref{eqn13}), we conclude that 
$\operatorname{aff}(\{{\bf p}_{i_0},\dots, {\bf p}_{i_d}\})= \mathbb{R}^d$, i.e.,
that the simlpex with the verticies ${\bf p}_{i_0},\dots, {\bf p}_{i_d}$ is
nondegenerate or, what is the same, $\operatorname{acmd}(Z_{I_*})\neq 0$.
Then (\ref{eqn11}) implies $\operatorname{acmd}(Z'_{I_*})\neq 0$.
Together with (\ref{eqn8}), (\ref{eqn10}), and (\ref{eqn11}), this means  
that the semimetric space $(K_{|G|},z')$
satisfies the conditions (i)--(iv) of Theorem~\ref{thrm1}.
Thus, there is its isometric embedding ${\bf p'}: K_{|G|}\to \mathbb{R}^d$,
${\bf p'}_{\!\! i}={\bf p'}(i)$, $i\in V(K_{|G|})=V(G)$.

Denote by $A: \mathbb{R}^d\to \mathbb{R}^d$ the affine transformation such that
$A({\bf p}_{i_r})={\bf p'}_{\!\! {i_r}}$ for every $i_r \in I_*$, i.e., the affine 
transformation that maps the nondegenerate simplex with the verticies 
${\bf p}_{i_0},\dots, {\bf p}_{i_d}$ to the nondegenerate simplex with the verticies 
${\bf p'}_{\!\! {i_0}},\dots, {\bf p'}_{\!\! {i_d}}$.
It remains for us to prove that the same mapping $A$ takes ${\bf p}_j$ to 
${\bf p'}_{\!\! j}$ for every $j\in V(G)$.

Suppose $j\notin I_*=\{i_0,\dots,i_d\} \subset V(G)$ and suppose $r=0,1,\dots, d$.
Denote by $\pi_r$ the affine hull of the set 
$\{{\bf p}_{i_0},\dots, {\bf p}_{i_d}\}\setminus\{{\bf p}_{i_r}\}$;
denote by $\pi'_r$ the affine hull of the set 
$\{{\bf p'}_{\!\! i_0},\dots, {\bf p'}_{\!\! i_d}\}\setminus\{{\bf p'}_{\!\! i_r}\}$;
and denote by $\operatorname{vol}_{d-1} (S)$ and $\operatorname{vol}_{d} (S)$ 
the $(d-1)$-volume and $d$-volume the set $S\subset \mathbb{R}^d$ respectively.

If ${\bf p}_{j}$ and ${\bf p}_{i_r}$ lie in the same open half-space defined by $\pi_r$,
denote by $\widetilde{\pi}'_r$ the hyperplane which is parallel to $\pi'_r$ and
lies in the half-space defined by $\pi'_r$ containing ${\bf p}_{i_r}$ at the distance
$$
h'_r\stackrel{\textrm{def}}{=} d\frac{\operatorname{vol}_d (
\operatorname{conv}(\{{\bf p'}_{\!\! j}, {\bf p'}_{\!\! i_0},\dots, 
{\bf p'}_{\!\! i_d}\}\setminus \{{\bf p'}_{\!\! i_r}\}))}{\operatorname{vol}_{d-1} 
(\operatorname{conv}(\{{\bf p'}_{\!\! i_0},\dots, {\bf p'}_{\!\! i_d}\}\setminus 
\{{\bf p'}_{\!\! i_r}\})}
$$
from $\pi'_r$.
Here $\operatorname{conv}(T)$ stands for the convex hull of the set 
$T\subset \mathbb{R}^d$.
If ${\bf p}_{j}\in \pi_r$, we put by definition $\widetilde{\pi}'_r=\pi'_r$.
If ${\bf p}_{j}$ and ${\bf p}_{i_r}$ lie in different open half-spaces defined by 
$\pi_r$, denote by $\widetilde{\pi}'_r$ the hyperplane which is parallel to $\pi'_r$ and
lies in the half-space defined by $\pi'_r$ which does not contain ${\bf p}_{i_r}$ at
distance $h'_r$ from $\pi'_r$.

Using (\ref{eqn2}) and (\ref{eqn10})--(\ref{eqn12}), it is straightforward to prove that 
${\bf p'}_{\!\! j}\in \widetilde{\pi}'_r$ and $A({\bf p}_j)\in \widetilde{\pi}'_r$.
Hence, ${\bf p'}_{\!\! j}\in \cup_{r=0}^d\widetilde{\pi}'_r$ 
and $A({\bf p}_j)\in \cup_{r=0}^d\widetilde{\pi}'_r$.
Since $\widetilde{\pi}'_r$ is parallel to $\pi'_r$ and 
$\pi'_r$ is the affine hull of a facet of a nondegenerate simplex in $\mathbb{R}^d$,
$\cup_{r=0}^d\widetilde{\pi}'_r$ contains at most one point.
Thus $A({\bf p}_j)={\bf p'}_{\!\! j}$ for all $j\in V(G)$. 
This means that the above constructed isometric embeddings ${\bf p}$ and
${\bf p'}$ are affine-equivalent.
Thus the implication $(\beta)\Rightarrow (\alpha)$ is proved.
\end{proof}

\section{Final remarks}\label{sec5}

Theorem~\ref{thrm2} allows us to solve not only Problem~\ref{probl1}, 
but also a number of related problems.
For example, in order to solve Problem~\ref{probl1} without condition (d) 
we should apply Theorem~\ref{thrm2} to every $\mathbb{R}^m$, $1\leqslant m<d$.
Another example is when we are given a framework $G(\bf{p})$ and want 
to know if there exists another geometric realization of the same graph $G$ which has
prescribed edge lengths $\lambda'_{ij}$ and is affine-equivalent to $G(\bf{p})$.
In this case, according to Theorem~\ref{thrm2} we should solve the 
system (\ref{eqn6})--({\ref{eqn12}) treating  $z_{ij}$ as prescribed real numbers
produced by $G(\bf{p})$.

According to Theorem~\ref{thrm2}, the positive solution of Problem~\ref{probl1} 
is equivalent to the existence of a solution to the system of polynomial equalities 
and inequalities (\ref{eqn6})--(\ref{eqn12}) in the set of real numbers.
It follows from the Tarski--Seidenberg Theorem \cite[Corollary 3.3.18]{Ma02} 
that unkown variables $z_{ij}$ and $z'_{ij}$, $(i,j)\notin E(G)$,
can be eliminated from (\ref{eqn6})--(\ref{eqn12}) 
so that the solvability of this system is equivalent to the validity of some 
quantifier-free Boolean formula whose atomic formulas are polynomial equations and inequalities containing only the lengths of bars 
$\lambda_{ij}$ and $\lambda'_{ij}$, $(i,j)\in E(G)$.
For this reason, we say that theoretically Problem~\ref{probl1} is always solvable.
However, the computational complexity of currently known algorithms for elimination 
of variables is so high that they cannot be used in practice. 
For this reason, we say that we cannot propose any practical algorithm for solution
of Problem~\ref{probl1}.

\subsection*{Acknowledgement}
The work was carried out in the framework of the State Task to the Sobolev 
Institute of Mathematics (Project FWNF-2022-0006).

\end{document}